\documentclass[12pt,reqno]{amsart}
\usepackage{amsmath,amssymb,amsthm}
\usepackage{mathrsfs}
\usepackage{multirow}
\usepackage{tikz}
\usetikzlibrary{arrows.meta,bending}
\usepackage[margin=1in]{geometry}
\usepackage[colorlinks=true,linkcolor=blue,citecolor=blue,urlcolor=blue]{hyperref}

\theoremstyle{plain}
\newtheorem{theorem}{Theorem}[section]
\newtheorem{proposition}[theorem]{Proposition}
\newtheorem{lemma}[theorem]{Lemma}
\newtheorem{corollary}[theorem]{Corollary}
\theoremstyle{definition}
\newtheorem{definition}[theorem]{Definition}
\newtheorem{remark}[theorem]{Remark}
\newtheorem{example}[theorem]{Example}

\newcommand{\Z}{\mathbb{Z}}
\newcommand{\T}{\mathbb{T}}

\newcommand{\coker}{\operatorname{coker}}
\newcommand{\Ot}{\mathcal{O}}

\begin{document}

\title[$C^*$-algebras of locally finite undirected graphs: K-theory]{The $C^*$-algebras of
locally finite undirected graphs: A complete description of their K-theory}

\author{D. Pask}
\address{College of Science and Engineering, James Cook University, Townsville QLD 4811, Australia}
\email{david.a.pask@gmail.com}
\date{\today}

\thanks{This research was partially funded by ARC DP\,150101598.}

\begin{abstract}
We study the $C^*$-algebra $C^*(\Gamma)$ of a locally finite undirected (Serre) graph $\Gamma$ and compute its
K-theory. The algebra is defined intrinsically, as the graph-of-groups algebra with all groups
trivial, and is shown to be independent of the choice of orientation. Its structure is accessed
through a passage to directed graphs: for every locally finite $\Gamma$ there is a row-finite
directed graph $E_\Gamma$ with $C^*(\Gamma)\cong C^*(E_\Gamma)$, and for essential $\Gamma$ this
identifies $C^*(\Gamma)$ with the Cuntz--Krieger algebra of the Bass--Hashimoto (non-backtracking)
matrix $T_\Gamma$. This makes the whole directed-graph K-theory machinery available and, unlike the
directed case, produces K-groups that read off the geometry of $\Gamma$: for a finite essential graph
of genus $g\ge 2$ one has $K_0(C^*(\Gamma))\cong\Z^{g}\oplus\Z/(g-1)\Z$ and
$K_1(C^*(\Gamma))\cong H_1(\Gamma)\cong\Z^{g}$, and in general $K_\ast$ is governed by the genus, the
number of ends and the number of dead-ends. We record the resulting classification by genus: the algebras are finite-dimensional or AF at genus
$0$, AT precisely for finite graphs at genus $1$, and unital Kirchberg algebras for finite essential
graphs of genus $\ge 2$.
\end{abstract}

\maketitle

\section{Introduction}\label{sec:intro}

The K-theory of the $C^*$-algebra of a directed graph is dynamical in character: for a finite
graph it is the Bowen--Franks data of the associated shift of finite type, and it bears little
relation to the geometry of the underlying graph. The starting point of this paper is the
observation that for \emph{undirected} graphs the situation is entirely different. To a locally
finite graph $\Gamma$ in the sense of Serre we associate an intrinsically defined $C^*$-algebra
$C^*(\Gamma)$---the graph-of-groups algebra of \cite{BMPST} with all vertex and edge groups
trivial---and its K-theory turns out to be a direct transcription of the geometry of $\Gamma$: the
genus $g$ \textup(first Betti number\textup), the graph valency $\gamma$ \textup(which, when
finite, counts the infinite chains\textup), and the number $d$ of dead-ends. The first computations of this kind are due to
Cornelissen, Lorscheid and Marcolli \cite{CLM} for finite graphs without sinks, and to Iyudu
\cite{Iyudu} for infinite graphs of finite Betti number; the purpose of this paper is to place
those computations in an intrinsic framework and to complete the classification across all genera,
including the non-essential graphs and the extreme genera $g=0,1$ lying outside the scope of
\cite{CLM,Iyudu}.

Our techniques are those of directed graph $C^*$-algebras, organised around a single construction:
to every locally finite $\Gamma$ we associate a row-finite directed graph $E_\Gamma$, built from
the reduced paths of length two with a tail adjoined at each dead-end, and prove that
$C^*(\Gamma)\cong C^*(E_\Gamma)$ \textup(Theorem~\ref{thm:general-realisation}\textup). For
essential $\Gamma$ the graph $E_\Gamma$ is the undirected analogue of the dual graph of Bates
\cite{Bates}, and the realisation identifies $C^*(\Gamma)$ with the Cuntz--Krieger algebra of the
Bass--Hashimoto \textup(non-backtracking\textup) operator $T_\Gamma$
\textup(Theorem~\ref{thm:essential-CK}\textup). The K-theoretic computations then combine several
ingredients: an embedding of the cycle space $H_1(\Gamma)$ into $\ker(1-T_\Gamma^{t})$ following
Robertson \cite{Robertson}; a contraction of a spanning tree onto a bouquet, implemented by an
explicit change of basis and justified at the $C^*$-level by the collapse move of S\o rensen
\cite{Sorensen}; a systematic bookkeeping of the sink and infinite-chain classes in the
presentation of $K_0$ coming from \cite{RS}, which produces the uniform formula
$K_0(C^*(\Gamma))\cong\Z^{g+\gamma+d}$ once $\Gamma$ fails to be essential and identifies the
torsion class $\Z/(g-1)\Z$ as the degeneration $\gamma=d=0$; and, for the type of the algebras, the
permanence properties of AT-algebras from \cite{Rordam}, whose failure under extensions---the
Toeplitz phenomenon---shows that the infinite genus-one algebras are not AT.

We work throughout with locally finite graphs, so that $E_\Gamma$ is row-finite. We do not treat
graphs with a vertex of infinite valence: Theorem~\ref{thm:general-realisation} and the results
built on it use the row-finite technology---the gauge-invariant uniqueness theorem and the
K-theory recipes of \cite{RS}---and in the non-row-finite setting they would require new proofs,
presumably through the Cuntz--Krieger algebras of infinite matrices and desingularisation. The
likelihood is that the geometric content of the K-theory would persist, with many of the K-groups
infinitely generated \textup(a vertex of infinite valence produces infinitely many ends\textup);
we therefore omit this case.

The paper is organised as follows. Section~\ref{sec:graphs} recalls Serre's notion of an
undirected graph and the combinatorial invariants---genus, graph valency, dead-ends---that
organise the results. Section~\ref{sec:algebra} defines $C^*(\Gamma)$ by generators and relations
and records the gauge action and the independence of orientation. Section~\ref{sec:realisation}
constructs the directed graph $E_\Gamma$ \textup(Definition~\ref{def:EGamma}\textup), proves the
realisation $C^*(\Gamma)\cong C^*(E_\Gamma)$ \textup(Theorem~\ref{thm:general-realisation}\textup)
and, for essential $\Gamma$, the identification $C^*(\Gamma)\cong\Ot_{T_\Gamma}$ with the
Cuntz--Krieger algebra of the Bass--Hashimoto operator
\textup(Theorem~\ref{thm:essential-CK}\textup). Section~\ref{sec:ktheory} contains the K-theory
and the classification by genus: $K_1(C^*(\Gamma))\cong H_1(\Gamma)$ and
$K_0(C^*(\Gamma))\cong\Z^{g}\oplus\Z/(g-1)\Z$ for finite essential $\Gamma$ of genus $g\ge 2$
\textup(Theorems~\ref{thm:K1-homology} and~\ref{thm:K0-torsion}\textup); the genus-zero and
genus-one cases, where the algebras are finite-dimensional, AF, AT, or---in the infinite genus-one
case---contain an infinite projection \textup(Lemma~\ref{lem:branching-ends},
Theorems~\ref{thm:genus0} and~\ref{thm:genus1} and Corollary~\ref{cor:genus-one-nonessential}\textup); and the
non-essential graphs of genus $g\ge 2$, with $K_1\cong\Z^{g}$ and $K_0\cong\Z^{g+\gamma+d}$
\textup(Theorem~\ref{thm:nonessential-K}\textup), where $\gamma=\gamma(\Gamma)$ is the graph
valency of $\Gamma$ \textup(Definition~\ref{def:valency}\textup). Section~\ref{sec:examples}
illustrates the results with worked examples across the genera, and closes with a summary table
recording, in each case, the K-groups, the type of the algebra, and where the entry is proved.

\section{Undirected graphs}\label{sec:graphs}

Following Serre \cite{Serre}, an \emph{(undirected) graph} is a quadruple
$\Gamma=(\Gamma^0,\Gamma^1,r,s)$ of countable vertex and edge sets with range and source maps
$r,s\colon\Gamma^1\to\Gamma^0$ and an edge-reversing involution $e\mapsto\bar e$ satisfying
\[
\bar e\neq e,\qquad \bar{\bar e}=e,\qquad s(\bar e)=r(e)\qquad(e\in\Gamma^1).
\]
Each geometric edge is thus a pair $\{e,\bar e\}$; drawing $\Gamma$ means choosing one of each pair.
The \emph{valence} of $v$ is $|r^{-1}(v)|=|s^{-1}(v)|$, and $\Gamma$ is \emph{locally finite} if every
valence is finite---the standing hypothesis throughout. We call $\Gamma$ \emph{essential} if
$|s^{-1}(v)|\ge 2$ for all $v$ (equivalently every valence is at least $2$). A vertex $v \in \Gamma^0$ is a sink if and only if it has valency one, sometimes we say that such a vertex is a \emph{dead-end}. This use of \emph{essential} is intrinsic to undirected graphs and should be distinguished from the established usage for directed graphs, in which a directed graph is called essential when it has no sources and no sinks; whenever the directed notion is meant below we say so explicitly (see Remark~\ref{rem:EGamma}).

A \emph{path} is a word
$e_n\cdots e_1$ with $s(e_{i+1})=r(e_i)$; it is \emph{reduced} if $e_{i+1}\neq\bar e_i$ for all $i$
(no immediate back-tracking), and its \emph{reversal} is $\bar e_1\cdots\bar e_n$. Note we use the ``Australian'' path convention here, in order to connect seamlessly with the established directed-graph literature, consult \cite{Raeburn}. We consider only connected graphs, that is for every $u,v \in \Gamma^0$ there is a path $e_n  \cdots e_1$ with $r(e_n)=v$ and $s(e_1)=u$.

Two combinatorial invariants organise everything below. The \emph{genus} (first Betti number)
$g=g(\Gamma)$ is the rank of $H_1(\Gamma)$, the number of geometric edges outside a spanning tree;
$g=0$ exactly when $\Gamma$ is a tree; a finite tree cannot be essential. For an infinite graph of
finite genus, which looks like a finite graph with finitely or infinitely many trees hanging off, the
\emph{graph valency} $\gamma=\gamma(\Gamma)$ is the sum of the valencies at the root and branching
vertices of those trees, in the following sense. Suppose we fix a finite subgraph of $\Gamma$ with
Betti number $g(\Gamma)$ \textup(the definition does not depend on this choice\textup). A
\emph{root vertex} is a root of one of the infinite trees; its valency is the number of outgoing
edges continuing to infinity. A \emph{branching vertex} is a vertex of a tree with one incoming
edge and more than one outgoing edge which leads to an infinite path; if $n$ is the number of such
outgoing edges, the valency of that vertex is $n-1$.

The definition below is taken from \cite[Definition~1.4]{Iyudu} and Example~\ref{ex:binary-tree}
below is \cite[Example~1]{Iyudu}.

\begin{definition}\label{def:valency}
The \emph{graph valency} $\gamma(\Gamma)$ of an infinite, locally finite graph $\Gamma$ of finite
genus $g$ is the sum of the valencies over all branching vertices and root vertices.
\end{definition}

If the graph valency is finite, then the graph valency coincides with the number of infinite chains outgoing from
the finite subgraph of $\Gamma$ \textup(that is, with the number of infinite ends\textup); in
general, however, it does not, as the following examples demonstrate.

\begin{example}\label{ex:tripod-gamma}
Consider the graph $\Gamma$ shown below, consisting of three infinite rays $(e_i)$, $(f_i)$ and
$(g_i)$ attached at a vertex $v$ \textup(reversed edges omitted\textup).
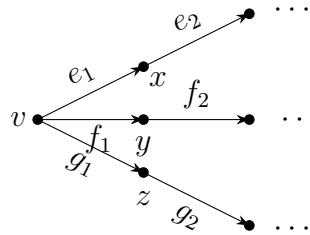
\begin{center}
\begin{tikzpicture}[>=Stealth]
\fill (0,0) circle (2pt) node[left] {$v$};
\draw[->] (0,0) -- node[above,sloped] {$e_1$} (1.4,0.7);
\draw[->] (1.4,0.7) -- node[above,sloped] {$e_2$} (2.8,1.4);
\fill (1.4,0.7) circle (2pt) node[below right=-2pt] {$x$}; \fill (2.8,1.4) circle (2pt);
\node at (3.4,1.45) {$\cdots$};
\draw[->] (0,0) -- node[below,yshift=2pt,xshift=3pt] {$f_1$} (1.4,0);
\draw[->] (1.4,0) -- node[above] {$f_2$} (2.8,0);
\fill (1.4,0) circle (2pt) node[below=2pt] {$y$}; \fill (2.8,0) circle (2pt);
\node at (3.5,0) {$\cdots$};
\draw[->] (0,0) -- node[below,sloped] {$g_1$} (1.4,-0.7);
\draw[->] (1.4,-0.7) -- node[below,sloped] {$g_2$} (2.8,-1.4);
\fill (1.4,-0.7) circle (2pt) node[below=3pt] {$z$}; \fill (2.8,-1.4) circle (2pt);
\node at (3.4,-1.45) {$\cdots$};
\end{tikzpicture}
\end{center}
We choose the finite subgraph $E$ to be the vertices $v,x,y$ and $z$, with edges $e_1$, $f_1$,
$g_1$. With root vertex $v$ we see that the root number is $3$, and since there are no branching
vertices we have $\gamma(\Gamma)=3$, with corresponding infinite paths $(e_i)_{i=1}^{\infty}$,
$(f_i)_{i=1}^{\infty}$ and $(g_i)_{i=1}^{\infty}$.
\end{example}

\begin{example}\label{ex:binary-tree}
Consider the full binary tree $BT$. The graph valency $\gamma(BT)$ is countable: all its vertices
are branching vertices of valency one, so the graph valency is equal to the number of vertices of
the tree. The infinite ends, however, can be enumerated by all sequences of $0$s and $1$s: if for
each vertex we label the edge going to the right by $0$ and the edge going to the left by $1$,
then the infinite paths are marked by all $0$--$1$ sequences, so there is a continuum of them.
\end{example}

\subsection*{Earlier results} The K-theory of $C^*(\Gamma)$ was first computed, for finite graphs, by
Cornelissen, Lorscheid and Marcolli \cite{CLM}, and extended to infinite locally finite graphs by
Iyudu \cite{Iyudu}; both proceed through the Bass--Hashimoto operator of \S\ref{sec:realisation} and the Cuntz--Krieger algebras of \cite{CK}. For
a finite connected graph $\Gamma$ of genus $g$ with no sinks, \cite{CLM} prove
\[
K_0(C^*(\Gamma))\cong\Z^{g}\oplus\Z/(g-1)\Z\ \ (g\ge 1),\qquad
K_1(C^*(\Gamma))\cong H_1(\Gamma)\cong\Z^{g}\ \ (g\ge 2),
\]
identifying $K_1$ with the cycle space $H_1(\Gamma)$, and locate the class of the unit in $K_0$ as an
element of order $(g-1)/\gcd(g-1,|\Gamma^0|)$. Iyudu \cite{Iyudu} treats the infinite locally finite
case of finite Betti number $2\le g<\infty$: the torsion of $K_0$ then vanishes and the graph valency $\gamma$ appears in its place,
\[
K_0(C^*(\Gamma))\cong\Z^{g+\gamma},\qquad K_1(C^*(\Gamma))\cong\Z^{g},
\]
so that, unlike in the finite case, $K_1$ need no longer be the torsion-free part of $K_0$. Both
computations are recovered below from the directed-graph realisation
(Theorems~\ref{thm:K0-torsion} and~\ref{thm:K1-homology}), which moreover reaches the genera
$g=0,1$ lying outside their scope.

\section{The $C^*$-algebra of an undirected graph}\label{sec:algebra}

In contrast to \cite{CLM,Iyudu}, the algebra $C^*(\Gamma)$ is defined directly, as the graph-of-groups $C^*$-algebra of
\cite{BMPST} in which every vertex and edge group is trivial.

\begin{definition}\label{def:CG}
For a locally finite graph $\Gamma$, $C^*(\Gamma)$ is the universal $C^*$-algebra generated by
mutually orthogonal projections $\{p_u:u\in\Gamma^0\}$ and partial isometries $\{s_f:f\in\Gamma^1\}$
satisfying
\begin{itemize}
\item[(R1)] $p_u=s_f^{*}s_f+s_{\bar f}s_{\bar f}^{*}$ for every $f\in s^{-1}(u)$, $u\in\Gamma^0$;
\item[(R2)] $\displaystyle s_e^{*}s_e=\!\!\sum_{r(f)=s(e),\,f\neq\bar e}\!\! s_fs_f^{*}$ for every
$e\in\Gamma^1$ with $|r^{-1}(s(e))|\ge 2$.
\end{itemize}
\end{definition}

We use the (``Australian") conventions of Raeburn's CBMS lecture notes \cite{Raeburn}, so that a directed
Cuntz--Krieger family has $s_e^{*}s_e=p_{s(e)}$ and $p_v=\sum_{r(e)=v}s_es_e^{*}$. From (R1) each
$s_f^{*}s_f+s_{\bar f}s_{\bar f}^{*}$ is a projection, whence $s_f^{*}s_f$ and
$s_{\bar f}s_{\bar f}^{*}$ are orthogonal and so $s_fs_{\bar f}=s_{\bar f}s_f=0$; combining (R1)
with (R2) at an edge $e$ with $|r^{-1}(s(e))|\ge 2$ gives
$p_{s(e)}=\sum_{r(f)=s(e)}s_fs_f^{*}$, the Cuntz--Krieger relation at $s(e)$. A short check
shows $C^*(\Gamma)$ is independent of the chosen edge
orientation. There is a gauge action $\gamma\colon\T\to\operatorname{Aut}C^*(\Gamma)$ with
$\gamma_z(s_f)=zs_f$ and $\gamma_z(p_u)=p_u$, deduced from the corresponding action for graphs of
groups \cite{BMPST}. That $C^*(\Gamma)$ is non-trivial follows from the realisation results of the
next section.

\section{Realisation as a directed graph algebra}\label{sec:realisation}

The engine of the theory is a passage from $\Gamma$ to a directed graph whose (directed) graph
algebra recovers $C^*(\Gamma)$; this imports the full apparatus of directed-graph $C^*$-algebras
\cite{KPR,RS}. The directed graph in question is built from the reduced paths of length two,
$\Gamma^2=\{fe: s(f)=r(e),\ f\neq\bar e\}$, with a tail adjoined at each dead-end.

\begin{definition}\label{def:EGamma}
Let $\Gamma$ be a locally finite graph. Define a directed graph $E_\Gamma$ by
\begin{align*}
E_\Gamma^0 &= \{f : f\in\Gamma^1\}\cup\{u_f : f\in\Gamma^1,\ |r^{-1}(s(f))|=1\},\\
E_\Gamma^1 &= \{fe : fe\in\Gamma^2\}\cup\{w_f : f\in\Gamma^1,\ |r^{-1}(s(f))|=1\},
\end{align*}
with range and source maps $r_P(fe)=f$, $s_P(fe)=e$ and $r_P(w_f)=f$, $s_P(w_f)=u_f$.
\end{definition}

\begin{remark}\label{rem:EGamma}
Local finiteness of $\Gamma$ ensures that $E_\Gamma$ is row-finite. There is at most one edge of
$E_\Gamma$ between any two of its vertices: the only possible edge from $e$ to $f$ is $fe$, and
$w_f$ is the only edge issuing from $u_f$. The vertices $u_f$, one for
each edge $f$ issuing from a dead-end of $\Gamma$, are sources of $E_\Gamma$ carrying the
projections $s_f^*s_f$ left unresolved by (R2); dually, each vertex $e$ with $r(e)$ a dead-end is a
sink of $E_\Gamma$. If $\Gamma$ is essential then no dead-ends occur, so $E_\Gamma^0=\Gamma^1$ and
$E_\Gamma^1=\Gamma^2$, and $E_\Gamma$ is the undirected analogue of the dual graph of Bates
\cite[Theorem~3.1]{Bates}. In this case $E_\Gamma$ is
essential as a directed graph: essentiality of $\Gamma$ supplies, for each $e\in\Gamma^1$, edges
$fe$ and $eg$ of $E_\Gamma$, so $E_\Gamma$ has no sources and no sinks. Since $E_\Gamma$ has at
most one edge between any pair of vertices, we may work, without loss of generality, with graphs
having at most one edge between any two vertices.
\end{remark}

Non-essential graphs are handled by the tails adjoined at dead-ends, generalising \cite{BMPST}; the
realisation theorem holds in full generality.

\begin{theorem}\label{thm:general-realisation}
For every locally finite graph $\Gamma$ we have $C^*(\Gamma)\cong C^*(E_\Gamma)$.
\end{theorem}

\begin{proof}[Sketch of proof]
Let $\{p_u,s_f\}$ be a universal generating family for $C^*(\Gamma)$. For $f\in\Gamma^1\subset
E_\Gamma^0$ put $Q_f=s_fs_f^*$; for each tail vertex put $Q_{u_f}=s_f^*s_f$, the projection left
unresolved by (R2); for $fe\in\Gamma^2$ put $S_{fe}=s_fs_es_e^*$; and for each tail edge put
$T_{w_f}=s_f$. Relations (R1) and (R2) show that $\{Q_v : v\in E_\Gamma^0\}$ is a family of
mutually orthogonal projections. Since $fe\in\Gamma^2$ guarantees that (R2) applies at $f$, we have
$s_es_e^*\le s_f^*s_f$ whenever $fe\in\Gamma^2$, whence
$S_{fe}^*S_{fe}=s_es_e^*=Q_{s_P(fe)}$; and for $f$ with $|r^{-1}(s(f))|\ge 2$, summing over the
edges of $E_\Gamma$ with range $f$,
\[
\sum_{r(e)=s(f),\,e\neq\bar f}S_{fe}S_{fe}^{*}
= s_f\Big(\sum_{r(e)=s(f),\,e\neq\bar f}s_es_e^{*}\Big)s_f^{*}
= s_f\,s_f^{*}s_f\,s_f^{*}=Q_f,
\]
again by (R2). For the tail edges,
$T_{w_f}^*T_{w_f}=s_f^*s_f=Q_{u_f}=Q_{s_P(w_f)}$, and since $w_f$ is the unique edge of $E_\Gamma$
with range $f$, the relation at $f$ is $T_{w_f}T_{w_f}^*=s_fs_f^*=Q_f$; the vertices $u_f$ are
sources and carry no relation. Hence $\{Q_v,S_{fe},T_{w_f}\}$ is a Cuntz--Krieger
$E_\Gamma$-family in $C^*(\Gamma)$, and the universal property of $C^*(E_\Gamma)$ yields a
homomorphism $\pi\colon C^*(E_\Gamma)\to C^*(\Gamma)$ intertwining the gauge actions. Since each
$Q_v\neq 0$, the gauge-invariant uniqueness theorem shows $\pi$ is injective. For surjectivity,
$s_f=T_{w_f}$ when $|r^{-1}(s(f))|=1$, while otherwise (R2) gives
$s_f=s_fs_f^*s_f=\sum_{r(e)=s(f),\,e\neq\bar f}S_{fe}$; the projections $p_u$ then lie in the image
by (R1).
\end{proof}

\subsection*{The essential case and the Bass--Hashimoto operator} For an essential locally finite
$\Gamma$ we may realise $C^*(\Gamma)$ as a Cuntz--Krieger algebra through the
non-backtracking operator $T_\Gamma$.

\begin{definition}\label{def:BH}
For an essential locally finite $\Gamma$, the \emph{Bass--Hashimoto operator}
$T_\Gamma\colon\Z\Gamma^1\to\Z\Gamma^1$ is
\[
T_\Gamma(h)(e)=\!\!\sum_{r(f)=s(e),\,f\neq\bar e}\!\! h(f),\qquad\text{equivalently}\qquad
T_\Gamma\,\delta_e=\!\!\sum_{r(f)=s(e),\,f\neq\bar e}\!\!\delta_f,
\]
\noindent
where $\{ \delta_e : e \in \Gamma^1 \}$ is a basis for $\mathbb{Z} \Gamma^1$. Then $T_\Gamma$ is
a $\Gamma^1\times\Gamma^1$ $\{0,1\}$-matrix whose $e$-column has a $1$ in row $f$ if and only if
$ef$ is a reduced path (that is, $r(f)=s(e)$ and $f\neq\bar e$). The sums are taken over
$r(f)=s(e)$, rather than over $s(f)=r(e)$, to fit with the Australian path conventions in use
throughout the paper; in particular relation (R2) is then literally the Cuntz--Krieger relation for
$T_\Gamma$ (see Theorem~\ref{thm:essential-CK}). The transpose of $T_\Gamma$ is
$T_\Gamma^{t}(h)(e)=\sum_{s(f)=r(e),\,f\neq\bar e}h(f)$, whose fixed-point space, for connected
essential $\Gamma$ of genus $g\ge 2$, is the cycle space $H_1(\Gamma)$
\textup(Theorem~\ref{thm:K1-homology}\textup).
\end{definition}

\begin{theorem}\label{thm:essential-CK}
Let $\Gamma$ be an essential locally finite graph. Then $C^*(\Gamma)\cong\Ot_{T_\Gamma}$, the
Cuntz--Krieger algebra of the Bass--Hashimoto matrix.
\end{theorem}

The proof uses that essentiality makes (R2) available at every vertex; then (R2) is precisely the
Cuntz--Krieger relation $s_e^{*}s_e=\sum_f T_\Gamma(f,e)\,s_fs_f^{*}$ for the matrix $T_\Gamma$, so
$\{s_e,p_u\}$ is a Cuntz--Krieger $T_\Gamma$-family, and the gauge-invariant uniqueness theorem gives
the isomorphism.

\section{K-theory and the classification by genus}\label{sec:ktheory}

Because $C^*(\Gamma)$ is a directed-graph algebra, its K-theory is computed by the standard
recipe. The point is that the answer is \emph{geometric}: it reads off the genus and the ends of
$\Gamma$, in contrast to directed graph algebras, whose K-theory is the (dynamical) Bowen--Franks
data of a shift of finite type and bears little relation to the underlying geometry.

\begin{corollary}[to Theorem~\ref{thm:essential-CK}]\label{cor:ker-coker}
For an essential locally finite graph $\Gamma$,
\[
K_0(C^*(\Gamma))\cong\coker\!\big(1-T_\Gamma^{t}\big),\qquad
K_1(C^*(\Gamma))\cong\ker\!\big(1-T_\Gamma^{t}\big).
\]
If $\Gamma$ is finite, connected and of genus $g\ge 2$, then $C^*(\Gamma)$ is purely infinite and
simple.
\end{corollary}

The K-theory formula is the standard one for Cuntz--Krieger and graph $C^*$-algebras \cite{PR,RS},
and the pure infiniteness and simplicity follow because $T_\Gamma$ is then irreducible and satisfies
condition~(I) \cite{CK}. The kernel is exactly the cycle space $H_1(\Gamma)$, which yields the free part
uniformly.

\subsection*{Genus zero} At genus $0$ and $1$ the algebra is not purely infinite, and, unlike the Cuntz--Krieger approach of \cite{CLM,Iyudu}, our
directed model for $C^*(\Gamma)$ identifies it precisely; we treat the genera in turn. The genus-zero analysis rests on the
notion of an end of a directed graph, taken from \cite[Definition~3.6]{PRennie}. Note that a finite
graph of genus zero is a tree and so cannot be essential, while an infinite essential graph of genus
zero is a tree in which all paths are infinite; the infinite binary tree is an example.

\begin{definition}[{\cite[Definition~3.6]{PRennie}}]\label{def:end}
Let $E$ be a row-finite directed graph. An \emph{end} of $E$ is a sink, a loop without exit, or an
infinite path with no exits.
\end{definition}

\begin{lemma}[{cf.\ \cite[Lemma~6.1]{PRennie}}]\label{lem:ends-ktheory}
Let $E$ be a row-finite directed graph in which no loop has an exit. Then
\[
K_0(C^*(E))\cong\Z^{\#\mathrm{ends}},\qquad K_1(C^*(E))\cong\Z^{\#\mathrm{loops}}.
\]
\end{lemma}

\begin{proof}
This follows from the continuity of $K_*$ and \cite[Corollary~5.3]{RS}.
\end{proof}

\begin{lemma}\label{lem:branching-ends}
Let $\Gamma$ be a \textup(not necessarily essential\textup) graph of genus zero with finitely many
dead-ends, $d$ in number.
\begin{enumerate}
\item If $\Gamma$ is finite, then it cannot be essential.
\item If $\Gamma$ is finite, then $E_\Gamma$ has $d$ ends, and hence
$K_0(C^*(\Gamma))\cong\Z^{d}$ and $K_1(C^*(\Gamma))=0$; moreover $C^*(\Gamma)$ is a direct sum of
$d$ matrix algebras, and so is finite-dimensional.
\item If $\Gamma$ is infinite with graph valency $\gamma(\Gamma)<\infty$, then $E_\Gamma$ has
$\gamma(\Gamma)+d$ ends, and hence $K_0(C^*(\Gamma))\cong\Z^{\gamma(\Gamma)+d}$ and
$K_1(C^*(\Gamma))=0$; in particular, if $\Gamma$ is essential then $d=0$ and
$K_0(C^*(\Gamma))\cong\Z^{\gamma(\Gamma)}$.
\end{enumerate}
\end{lemma}

\begin{proof}
Since $\Gamma$ has genus zero it is a tree. For (1), a finite tree has a vertex of valence one, so
cannot be essential. For (2) and (3), $E_\Gamma$ contains no loops, so
Lemma~\ref{lem:ends-ktheory} applies, and the ends of $E_\Gamma$ are of two kinds. First, each
dead-end $v$ of $\Gamma$ has a unique incoming edge, and by Remark~\ref{rem:EGamma} that edge is a
sink of $E_\Gamma$; distinct dead-ends give distinct sinks, and every sink of $E_\Gamma$ arises
this way, so the sinks number $d$. Second, the remaining ends of $E_\Gamma$ are its infinite paths
with no exits: taking the finite subgraph of the same genus in the definition of the graph valency to be a single vertex $v$, and comparing Definition~\ref{def:end} with that definition,
these correspond to the infinite chains of $\Gamma$ leaving $v$; the finite branches of $\Gamma$
contribute no infinite paths, only the sinks already counted. If $\Gamma$ is finite then every
path in $\Gamma$ is finite and the ends of $E_\Gamma$ are precisely its $d$ sinks; moreover
\cite[Corollary~2.3]{KPR} identifies $C^*(E_\Gamma)$ as the direct sum of one matrix algebra for
each sink \textup(cf.\ Example~\ref{ex:genus-zero-tripod}\textup), giving (2). If
$\Gamma$ is infinite the infinite chains number $\gamma(\Gamma)$, so $E_\Gamma$ has
$\gamma(\Gamma)+d$ ends, giving (3). In both cases Lemma~\ref{lem:ends-ktheory} gives the
K-groups.
\end{proof}

The type of the algebra is settled by the following observation.

\begin{proposition}\label{prop:AF}
Let $E$ be a row-finite directed graph with no cycles. If $E$ is finite then $C^*(E)$ is
finite-dimensional, and if $E$ is infinite then $C^*(E)$ is AF.
\end{proposition}

\begin{proof}[Sketch of proof]
We know from \cite[Theorem~2.4]{KPR} that the $C^*$-algebra of a directed graph with no cycles is
AF; when $E$ is finite, \cite[Corollary~2.3]{KPR} shows that $C^*(E)$ is a finite direct sum of
matrix algebras, and hence is finite-dimensional.
\end{proof}

\begin{theorem}\label{thm:genus0}
Let $\Gamma$ be locally finite, connected and essential. If $g(\Gamma)=0$ and $\Gamma$ is
infinite, then $C^*(\Gamma)$ is an AF algebra.
\end{theorem}

\begin{proof}[Sketch of proof]
An essential $\Gamma$ of genus zero is an infinite tree with all paths infinite, so
$E_\Gamma$ has no cycles and Proposition~\ref{prop:AF} applies.
\end{proof}

\subsection*{Genus one}

\begin{theorem}\label{thm:genus1}
Let $\Gamma$ be locally finite, connected and essential with $g(\Gamma)=1$.
\begin{enumerate}
\item If $\Gamma$ is finite, then $\Gamma$ is a circuit and
\[
C^*(\Gamma)\cong M_n(C(\T))\oplus M_n(C(\T))\ \text{for some $n$},\qquad K_0(C^*(\Gamma))\cong K_1(C^*(\Gamma))\cong\Z^{2},
\]
so $C^*(\Gamma)$ is an AT-algebra; in particular it is neither simple nor purely infinite even for
finite connected $\Gamma$.
\item If $\Gamma$ is infinite with graph valency $\gamma(\Gamma)<\infty$, then
\[
K_0(C^*(\Gamma))\cong\Z^{\gamma(\Gamma)+1},\qquad K_1(C^*(\Gamma))\cong\Z;
\]
moreover $C^*(\Gamma)$
contains an infinite projection, so unlike the finite case it is not an AT-algebra.
\end{enumerate}
\end{theorem}

\begin{proof}[Sketch of proof]
For part~(1), $T_\Gamma$ consists
of two cyclic permutations---one for
each way around the circuit---so $\Ot_{T_\Gamma}$ splits as two copies of $M_n(C(\T))$ for some
$n$. For part~(2), $\Gamma$ is a circuit with infinite trees attached, and $E_\Gamma$ has the
structure described before Example~\ref{ex:genus-one-infinite}. The kernel argument in the proof of
Theorem~\ref{thm:nonessential-K} applies with core the circuit, giving
$K_1(C^*(\Gamma))\cong H_1(\Gamma)\cong\Z$. For the cokernel, in the presentation ($**$) of that
proof the inward-chain classes are expressed in terms of the remaining generators, the classes
along each copy of the circuit reduce to a single generator, and traversing either copy of the
circuit produces the one relation $\sum_{l}O_l=0$ among the $\gamma$ outward-chain classes; hence
the cokernel is free on the two circuit classes and $\gamma-1$ of the $O_l$, giving
$K_0(C^*(\Gamma))\cong\Z^{\gamma+1}$ \textup(for $\gamma=1$ this is the computation of
Example~\ref{ex:genus-one-infinite}\textup). Finally, at a circuit vertex $x$ of $E_\Gamma$ to which an inward tree attaches, the
circuit edge $\mu$ satisfies $s_\mu^*s_\mu=p_x$ and $s_\mu s_\mu^*<p_x$, since $p_x$ also
dominates the range projection of the incoming tree edge; thus $s_\mu$ is a proper isometry in
$p_xC^*(\Gamma)p_x$ and $p_x$ is an infinite projection. Since every AT-algebra has stable rank
one, hence is stably finite \cite[Proposition~3.2.4]{Rordam}, and AT-algebras are not closed under
extensions \cite[Proposition~3.2.5]{Rordam}, the algebra $C^*(\Gamma)$ is not AT---the Toeplitz
extension is precisely the phenomenon occurring here.
\end{proof}

Theorem~\ref{thm:genus1}(1) supplies the case $g=1$ left open in \cite{CLM}.

The methods of Theorem~\ref{thm:genus1} also cover the non-essential genus-one graphs; the
non-essential genus-zero graphs are covered by Lemma~\ref{lem:branching-ends}.

\begin{corollary}\label{cor:genus-one-nonessential}
Let $\Gamma$ be a non-essential connected graph with $g(\Gamma)=1$, graph valency
$\gamma(\Gamma)<\infty$ and finitely many dead-ends, $d\ge 1$ in number \textup(if $\Gamma$ is
finite then $\gamma(\Gamma)=0$\textup). Then
\[
K_1(C^*(\Gamma))\cong H_1(\Gamma)\cong\Z,\qquad K_0(C^*(\Gamma))\cong\Z^{1+\gamma+d},
\]
and $C^*(\Gamma)$ contains an infinite projection, so is not an AT-algebra.
\end{corollary}

\begin{proof}[Sketch of proof]
Here $\Gamma$ is a circuit with trees attached, at least one of which carries a dead-end. The
arguments of Theorem~\ref{thm:genus1}(2) apply, with the presentation in the proof of
Theorem~\ref{thm:nonessential-K} handling the finite branches: each dead-end contributes its sink
class, and traversing either copy of the circuit produces the single relation
$\sum_jS_j+\sum_lO_l=0$, so the cokernel is free of rank $2+(d+\gamma)-1=1+\gamma+d$. For the
kernel, the bare circuit contributes $\Z^{2}$, cut down to $\Z\cong H_1(\Gamma)$ by the attachment
constraint at any attachment vertex, which forces the two constants to be opposite. The infinite
projection arises exactly as in Theorem~\ref{thm:genus1}(2), the inward copy of any attached tree
providing an entrance to each circuit. For example, for a finite circuit carrying two pendant
paths \textup($\gamma=0$, $d=2$\textup) a direct computation with the map $K$ of
\cite[Theorem~3.2]{RS} confirms $K_0\cong\Z^{3}$ and $K_1\cong\Z$.
\end{proof}

\subsection*{Genus at least two}

\begin{theorem}\label{thm:K1-homology}
For an essential locally finite connected graph $\Gamma$ of genus $g\ge 2$, the map
$\varphi_1(f)(x)=f(x)-f(\bar x)$ induces an isomorphism
\[
H_1(\Gamma)\cong\ker\!\big(1-T_\Gamma^{t}\big)\cong K_1(C^*(\Gamma))\cong\Z^{g}.
\]
\end{theorem}

\begin{proof}[Sketch of proof]
The argument follows closely that of Robertson \cite{Robertson}. Fix an orientation
$\Gamma^+\subset\Gamma^1$ containing exactly one edge from each pair $\{e,\bar e\}$, and regard
$\Z\Gamma^+$ as those functions on $\Gamma^1$ vanishing off $\Gamma^+$, so that
$H_1(\Gamma)=\ker\partial$ for the boundary map $\partial\colon\Z\Gamma^+\to\Z\Gamma^0$,
$\partial f(v)=\sum_{r(z)=v}f(z)-\sum_{s(y)=v}f(y)$. Define
$\varphi_0\colon\Z\Gamma^0\to\Z\Gamma^1$ by $\varphi_0f(x)=f(r(x))$. Writing
\[
\big(1-T_\Gamma^{t}\big)g(x)=g(x)+g(\bar x)-\sum_{s(y)=r(x)}g(y),
\]
a direct computation gives the intertwining relation
$\big(1-T_\Gamma^{t}\big)\varphi_1=\varphi_0\,\partial$, so $\varphi_1$ maps
$H_1(\Gamma)=\ker\partial$ into $\ker\big(1-T_\Gamma^{t}\big)$; it does so injectively, since
$\varphi_1f=0$ forces $f(x)=f(\bar x)$, and $f$ vanishes off $\Gamma^+$.

For surjectivity, let $g\in\ker\big(1-T_\Gamma^{t}\big)$ and set $\sigma(x)=g(x)+g(\bar x)$. The
displayed relation shows $\sigma(x)=\sum_{s(y)=r(x)}g(y)$ depends only on $r(x)$; since
$\sigma(x)=\sigma(\bar x)$, it also depends only on $s(x)$, and connectedness forces $\sigma$ to be
constant. If $\Gamma$ is finite, choosing one edge into each vertex and summing gives
$|\Gamma^0|\,\sigma=\sum_{x\in\Gamma^1}g(x)=\sum_{e\in\Gamma^+}\big(g(e)+g(\bar
e)\big)=|\Gamma^+|\,\sigma$; since $|\Gamma^+|=|\Gamma^0|+g-1>|\Gamma^0|$ when $g\ge 2$, we conclude
$\sigma=0$. Then $g(\bar x)=-g(x)$, and $f:=g|_{\Gamma^+}$ satisfies $\varphi_1f=g$ and, since
$\varphi_0$ is injective \textup(since $r\colon\Gamma^1\to\Gamma^0$ is surjective, every vertex of
the essential graph $\Gamma$ receiving an edge\textup), $\partial f=0$; thus
$g\in\varphi_1(H_1(\Gamma))$. If $\Gamma$ is
infinite, then $g$ has finite support and so is supported in a finite connected subgraph of
$\Gamma$; enlarging this to a finite connected subgraph $\widetilde\Gamma$ containing at least two
of the cycles generating $H_1(\Gamma)$ gives $g(\widetilde\Gamma)\ge 2$, and the finite argument
applies. Finally, $\ker\big(1-T_\Gamma^{t}\big)\cong K_1(C^*(\Gamma))$ by
Corollary~\ref{cor:ker-coker}, and $H_1(\Gamma)\cong\Z^{g}$ by the definition of the genus.
\end{proof}

Recall that the \emph{bouquet} (or \emph{rose}) with $g$ petals is the graph $B_g$ with a single
vertex and $g$ geometric loops, so $B_g^0=\{v\}$ and $B_g^1=\{e_1,\bar e_1,\dots,e_g,\bar e_g\}$;
it is essential of genus $g$ for $g\ge 1$. The proof of the next theorem reduces a finite essential
graph of genus $g$ to $B_g$.

\begin{theorem}\label{thm:K0-torsion}
For an essential connected graph $\Gamma$ with $2\le g<\infty$,
\[
K_0(C^*(\Gamma))\cong\Z^{g}\oplus\Z/(g-1)\Z .
\]
\end{theorem}

\begin{proof}[Sketch of proof]
We reduce $\Gamma$ to $B_g$ by successively deleting edges with distinct source and range. Fix
$e\in\Gamma^1$ with $s(e)\neq r(e)$, and let $\Sigma$ be the graph obtained from $\Gamma$ by
deleting $e$, $\bar e$ and the vertex $r(e)$, redirecting the severed edges to $s(e)$; then
$\Sigma$ is essential, connected, and of the same genus. Replacing the basis vectors $\delta_f$ of
$\Z\Gamma^1$ by $\delta_f+\delta_e$ or $\delta_f+\delta_{\bar e}$, as appropriate, for the edges
$f$ incident to the deleted vertex, a direct computation conjugates $1-T_\Gamma^{t}$ into
$\big(1-T_\Sigma^{t}\big)\oplus 1$, with the identity summand acting on the span of
$\delta_e,\delta_{\bar e}$; hence the kernel and cokernel are unchanged. At the level of the
directed graph, this contraction is implemented by the collapse move of S\o rensen
\cite[Theorem~5.2]{Sorensen}, applied at the regular vertices $e$ and $\bar e$ of $E_\Gamma$
\textup(neither supports a loop of length one, as $e$ is not a loop of $\Gamma$\textup), so that
$E_{\Sigma}$ is move-equivalent to $E_\Gamma$ and $C^*(\Sigma)$ is stably isomorphic to
$C^*(\Gamma)$; the change of basis is the matrix counterpart of the row and column additions of
\cite[Lemmas~7.1 and~7.2]{Sorensen}. Iterating over the
geometric edges of a spanning tree reduces $\Gamma$ to $B_g$.

For $B_g$ every pair of edges is composable, so
$1-T_{B_g}^{t}=1+P-\mathbf{1}\mathbf{1}^{t}$ on $\Z B_g^1\cong\Z^{2g}$, where $P$ is the
permutation $\delta_e\mapsto\delta_{\bar e}$ and $\mathbf{1}$ is the all-ones vector. Its image
lies in the sublattice $L\cong\Z^g$ of functions constant on the pairs $\{e,\bar e\}$, and under
the pair-sum identification $L\cong\Z^g$ the operator factors through $1_g-J_g$, where $J_g$ is
the all-ones $g\times g$ matrix. The invariant factors of $1_g-J_g$ are
$1,\dots,1,g-1$, so $L/\operatorname{im}\big(1-T_{B_g}^{t}\big)\cong\Z/(g-1)\Z$, while
$\Z^{2g}/L\cong\Z^{g}$ is free; the resulting extension splits, giving
$\coker\big(1-T_{B_g}^{t}\big)\cong\Z^{g}\oplus\Z/(g-1)\Z$. The theorem now follows from
Corollary~\ref{cor:ker-coker}.
\end{proof}

Theorems~\ref{thm:K1-homology} and~\ref{thm:K0-torsion} recover, and slightly extend, the
computations of Cornelissen--Lorscheid--Marcolli \cite{CLM} and Iyudu \cite{Iyudu}.

\begin{theorem}\label{thm:nonessential-K}
Let $\Gamma$ be a non-essential connected graph with $2\le g=g(\Gamma)<\infty$, graph valency
$\gamma=\gamma(\Gamma)<\infty$, and finitely many dead-ends, $d\ge 1$ in number. Then
\[
K_1(C^*(\Gamma))\cong H_1(\Gamma)\cong\Z^{g},\qquad K_0(C^*(\Gamma))\cong\Z^{g+\gamma+d}.
\]
Moreover $C^*(\Gamma)$ is not simple, but contains an infinite projection.
\end{theorem}

\begin{proof}[Sketch of proof]
Let $\Gamma_c$ be the smallest connected subgraph of $\Gamma$ containing every cycle. Since
$g<\infty$, $\Gamma_c$ is finite; it is essential of genus $g$, and $\Gamma\setminus\Gamma_c$ is a
disjoint union of trees, containing between them the $d$ dead-ends and the $\gamma$ infinite
chains. Write $W$ for the set of sinks of $E_\Gamma$, one for each dead-end, put
$V=E_\Gamma^0\setminus W$, and let $K\colon\Z V\to\Z V\oplus\Z W$ be the map
$K(x)=\big((1-B^{t})x,\,-C^{t}x\big)$ of \cite[Theorem~3.2]{RS}, where
$\left(\begin{smallmatrix} B & C\\ 0 & 0\end{smallmatrix}\right)$ is the block form of the vertex
matrix of $E_\Gamma$ with respect to $E_\Gamma^0=V\sqcup W$, so that
$K_1(C^*(\Gamma))\cong\ker K$ and $K_0(C^*(\Gamma))\cong\coker K$. Unwinding the definition of $K$, a
finitely supported $x\in\Z V$ lies in $\ker K$ if and only if
\begin{equation}
x_v=\sum_{\{w\in V:\,w\to v\}}x_w\quad(v\in V),\qquad\qquad
\sum_{\{v:\,v\to w\}}x_v=0\quad(w\in W),\tag{$*$}
\end{equation}
while $\coker K$ is generated by the classes $[u]$, $u\in E_\Gamma^0$, subject to
\begin{equation}
[v]=\sum_{\{u:\,v\to u\}}[u]\quad(v\in V).\tag{$**$}
\end{equation}

For $K_1$ we adapt the proof of Theorem~\ref{thm:K1-homology}. Let $x\in\ker K$. On each tree the
relations ($*$) force $x$ to vanish, working inwards from the extremities: the coordinates $u_f$
have no predecessors, so vanish; along an outward-pointing path from a dead-end each coordinate has
the previous one as its unique predecessor, so all vanish; the inward-pointing coordinates along
such a path are pairwise equal by ($*$), and equal to zero by the sink relation at the dead-end;
and along an infinite chain the inward and outward coordinates are constant, hence zero by finite
support \textup(cf.\ Example~\ref{ex:genus-two-nonessential}, where the coordinates at $u_e$ and $e$
vanish\textup). A general tree is handled by induction from its extremities. The relation ($*$) at
the innermost coordinate of each tree leaves behind the constraint
$\sum_{\{w\in\Gamma_c^1:\,r(w)=c\}}x_w=0$ at its attachment vertex $c$. Thus $x$ is supported on
$\Gamma_c^1$, where ($*$) is the kernel condition for $1-T^{t}_{\Gamma_c}$; by
Theorem~\ref{thm:K1-homology} applied to the finite essential graph $\Gamma_c$ we get
$x=\varphi_1(f)$ for a cycle $f$, and the attachment constraints are automatic on such $x$, since
$\sum_{r(w)=c}\big(f(w)-f(\bar w)\big)=-(\partial f)(c)=0$. Hence
$\ker K\cong H_1(\Gamma_c)=H_1(\Gamma)\cong\Z^{g}$.

For $K_0$ we first eliminate the tree generators using ($**$): $[u_f]=[f]$; the classes of the
outward-pointing edges of a tree are pairwise equal and equal to a sum of core classes at the
attachment vertex; the classes of the inward-pointing edges along a branch ending in a dead-end are
pairwise equal and equal to the class $S$ of the sink there; and along an infinite chain the
outward classes are pairwise identified to a single class $O$, which survives with no further
relation, by the continuity of $K_*$ as in Lemma~\ref{lem:ends-ktheory}. There results a
presentation of $\coker K$ with generators the core classes together with $S_1,\dots,S_d$ and
$O_1,\dots,O_\gamma$, in which each core relation acquires the attachment classes at its vertex as
extra summands. Contracting a spanning tree of $\Gamma_c$ as in the proof of
Theorem~\ref{thm:K0-torsion}---justified there by the collapse move of
\cite[Theorem~5.2]{Sorensen}---reduces to the bouquet $B_g$ carrying all the attachments at its
single vertex: writing $p_i=[x_i]+[\bar x_i]$ for the pair-sums of the $2g$ petal classes,
$q_i=[x_i]-[\bar x_i]$, and $T=\sum_jS_j+\sum_lO_l$, the relations become
\[
p_i=\sum_{j=1}^{g}p_j+T\qquad(1\le i\le g),
\]
so all $p_i$ equal a common class $p$ satisfying $(g-1)p+T=0$, while the differences $q_i$ are
unconstrained. When $d=\gamma=0$ this relation is $(g-1)p=0$ and we recover the torsion class of
Theorem~\ref{thm:K0-torsion}; here $d\ge1$, so the relation has a coefficient-one generator
\textup(any $S_j$\textup) and is unimodular: eliminating $S_1=(1-g)p-\sum_{j\ge2}S_j-\sum_lO_l$
leaves a free abelian group. Counting generators, $\coker K$ is free on the $g$ differences
$q_i$, the class $p$, the remaining $S_2,\dots,S_d$, and $O_1,\dots,O_\gamma$, giving
$\coker K\cong\Z^{g+\gamma+d}$. For $g=2$, $\gamma=0$, $d=1$ this is the computation of
Example~\ref{ex:genus-two-nonessential}.

Since $\Gamma$ has genus $g\ge 2$, there must be two circuits $\mu,\nu$ in $\Gamma$, which give
rise to four circuits $\mu,\bar\mu,\nu,\bar\nu$ in $E_\Gamma$. Since $\Gamma$ is connected, each of
these loops has an entrance, and so gives rise to an infinite projection by a standard argument.
These loops also cause $E_\Gamma$ to be non-cofinal, and so $C^*(\Gamma)\cong C^*(E_\Gamma)$ is not
simple.
\end{proof}

\section{Examples}\label{sec:examples}

\begin{example}\label{ex:genus-zero-tripod}
Recall the genus zero infinite graph $\Gamma$ given in Example~\ref{ex:tripod-gamma}, which is
essential. Now consider $E_\Gamma$, shown below: the inward chain
of each ray connects to the outward chains of the other two, and the three outward chains
$e_1\to e_2\to\cdots$, $f_1\to f_2\to\cdots$ and $g_1\to g_2\to\cdots$ are infinite paths with no
exits.
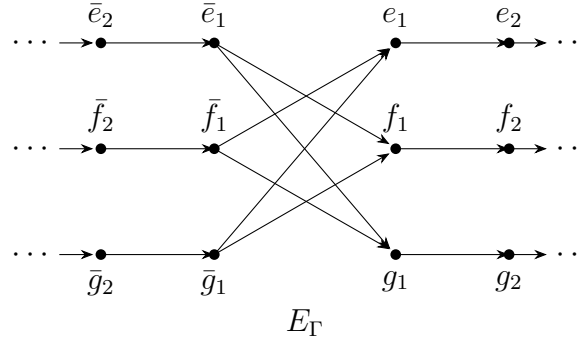
\begin{center}
\begin{tikzpicture}[>=Stealth, baseline=(current bounding box.center)]
\begin{scope}[xshift=2.4cm]
\fill (0,1.4) circle (2pt) node[above=2pt] {$\bar e_1$};
\fill (0,0) circle (2pt) node[above=2pt] {$\bar f_1$};
\fill (0,-1.4) circle (2pt) node[below=2pt] {$\bar g_1$};
\fill (-1.5,1.4) circle (2pt) node[above=2pt] {$\bar e_2$};
\fill (-1.5,0) circle (2pt) node[above=2pt] {$\bar f_2$};
\fill (-1.5,-1.4) circle (2pt) node[below=2pt] {$\bar g_2$};
\node at (-2.4,1.4) {$\cdots$};
\node at (-2.4,0) {$\cdots$};
\node at (-2.4,-1.4) {$\cdots$};
\draw[->] (-2.05,1.4) -- (-1.6,1.4);
\draw[->] (-2.05,0) -- (-1.6,0);
\draw[->] (-2.05,-1.4) -- (-1.6,-1.4);
\draw[->] (-1.5,1.4) -- (0,1.4);
\draw[->] (-1.5,0) -- (0,0);
\draw[->] (-1.5,-1.4) -- (0,-1.4);
\fill (2.4,1.4) circle (2pt) node[above=2pt] {$e_1$};
\fill (2.4,0) circle (2pt) node[above=2pt] {$f_1$};
\fill (2.4,-1.4) circle (2pt) node[below=2pt] {$g_1$};
\fill (3.9,1.4) circle (2pt) node[above=2pt] {$e_2$};
\fill (3.9,0) circle (2pt) node[above=2pt] {$f_2$};
\fill (3.9,-1.4) circle (2pt) node[below=2pt] {$g_2$};
\node at (4.8,1.4) {$\cdots$};
\node at (4.8,0) {$\cdots$};
\node at (4.8,-1.4) {$\cdots$};
\draw[->] (2.4,1.4) -- (3.9,1.4);
\draw[->] (2.4,0) -- (3.9,0);
\draw[->] (2.4,-1.4) -- (3.9,-1.4);
\draw[->] (3.9,1.4) -- (4.4,1.4);
\draw[->] (3.9,0) -- (4.4,0);
\draw[->] (3.9,-1.4) -- (4.4,-1.4);
\draw[->] (0,1.4) -- (2.32,0.06);
\draw[->] (0,1.4) -- (2.33,-1.33);
\draw[->] (0,0) -- (2.32,1.33);
\draw[->] (0,0) -- (2.32,-1.33);
\draw[->] (0,-1.4) -- (2.33,1.33);
\draw[->] (0,-1.4) -- (2.32,-0.06);
\node at (1.2,-2.3) {$E_\Gamma$};
\end{scope}
\end{tikzpicture}
\end{center}
Now $E_\Gamma$ has ends $(e_{i+1}e_i)_{i=1}^{\infty}$, $(f_{i+1}f_i)_{i=1}^{\infty}$ and
$(g_{i+1}g_i)_{i=1}^{\infty}$, so $K_0(C^*(\Gamma))\cong\Z^{3}=\Z^{\gamma(\Gamma)}$ by
Lemma~\ref{lem:branching-ends}(3). Since $E_\Gamma$ has no loops, $C^*(E_\Gamma)$ is AF by
\cite[Theorem~2.4]{KPR}, and $K_1(C^*(\Gamma))=0$ \textup(cf.\ Theorem~\ref{thm:genus0}\textup).

By way of contrast, consider the finite case: let $\Gamma'$ be the finite tripod $E$ of
Example~\ref{ex:tripod-gamma}, that is, the tree on the vertices $v,x,y,z$ with edges
$e_1,f_1,g_1$. Then $\Gamma'$ is non-essential with $d=3$ dead-ends, and $E_{\Gamma'}$ consists of
the same six hub edges as above, together with a tail at each of $\bar e_1,\bar f_1,\bar g_1$ and
sinks at $e_1,f_1,g_1$. There are no loops, and the ends of $E_{\Gamma'}$ are its three sinks, so
Lemma~\ref{lem:branching-ends}(2) gives $K_0(C^*(\Gamma'))\cong\Z^{3}=\Z^{d}$ and
$K_1(C^*(\Gamma'))=0$. In line with Proposition~\ref{prop:AF}, $C^*(\Gamma')$ is
finite-dimensional: since there is at most one directed path between any two vertices of
$E_{\Gamma'}$ \textup(Remark~\ref{rem:EGamma}\textup) and five vertices reach each sink,
$C^*(\Gamma')\cong M_5(\mathbb{C})\oplus M_5(\mathbb{C})\oplus M_5(\mathbb{C})$.
\end{example}

\begin{example}\label{ex:genus-one}
Let $\Gamma$ be the essential graph of genus one with two vertices and two geometric edges shown on
the left below, with reversed edges drawn dashed. According to Definition~\ref{def:EGamma} we have
$E_\Gamma^0=\{e,\bar e,f,\bar f\}$ and $E_\Gamma^1=\{fe,ef,\bar e\bar f,\bar f\bar e\}$, and
$E_\Gamma$ consists of the two disjoint circuits shown on the right, one for each way around
$\Gamma$.
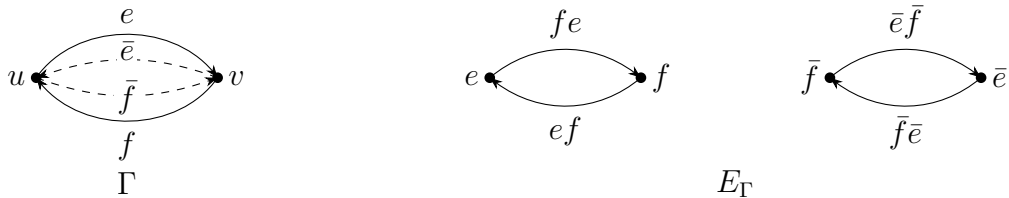
\begin{center}
\begin{tikzpicture}[>=Stealth, baseline=(current bounding box.center)]
\node at (1.2,-1.4) {$\Gamma$};
\fill (0,0) circle (2pt) node[left] {$u$};
\fill (2.4,0) circle (2pt) node[right] {$v$};
\draw[->] (0,0) to[bend left=55] node[above] {$e$} (2.4,0);
\draw[->,dashed] (2.4,0) to[bend right=20] node[below,pos=0.5,yshift=8pt,fill=white,inner sep=1pt] {$\bar e$} (0,0);
\draw[->] (2.4,0) to[bend left=55] node[below] {$f$} (0,0);
\draw[->,dashed] (0,0) to[bend right=20] node[above,pos=0.5,yshift=-8pt,fill=white,inner sep=1pt] {$\bar f$} (2.4,0);
\node at (9.25,-1.4) {$E_\Gamma$};
\begin{scope}[xshift=6cm]
\fill (0,0) circle (2pt) node[left] {$e$};
\fill (2,0) circle (2pt) node[right] {$f$};
\draw[->] (0,0) to[bend left=40] node[above] {$fe$} (2,0);
\draw[->] (2,0) to[bend left=40] node[below] {$ef$} (0,0);
\end{scope}
\begin{scope}[xshift=10.5cm]
\fill (0,0) circle (2pt) node[left] {$\bar f$};
\fill (2,0) circle (2pt) node[right] {$\bar e$};
\draw[->] (0,0) to[bend left=40] node[above] {$\bar e\bar f$} (2,0);
\draw[->] (2,0) to[bend left=40] node[below] {$\bar f\bar e$} (0,0);
\end{scope}
\end{tikzpicture}
\end{center}
In line with Theorem~\ref{thm:genus1}(1) we have $C^*(\Gamma)\cong C^*(E_\Gamma)\cong M_2(C(\T))\oplus M_2(C(\T))$, which is neither purely
infinite nor simple, even though $\Gamma$ is finite, connected and essential (cf.\
Corollary~\ref{cor:ker-coker}).
\end{example}

If $\Gamma$ is an infinite essential connected graph of genus one, then it consists of a circuit
with at least one infinite tree attached, all of whose paths are infinite. In this case $E_\Gamma$
consists of two copies of the circuit, one for each way around, together with two copies of each
attached tree: an inwards copy, whose edges point towards the circuits and which connects to both
copies, and an outwards copy, which both copies of the circuit connect to, as shown in the
following example.

\begin{example}\label{ex:genus-one-infinite}
Let $\Gamma$ be the genus one graph consisting of a single loop $e$ with an infinite ray $f_1f_2\cdots$
attached, shown on the left below (reversed edges omitted). Then $\Gamma$ is infinite and essential,
and $E_\Gamma$ is the directed graph shown on the right: the loops at $e$ and
$\bar e$ are the two ways around the circuit, the inwards copy
$\cdots\to\bar f_2\to\bar f_1$ of the ray connects to both, and both connect to the outwards copy
$f_1\to f_2\to\cdots$.
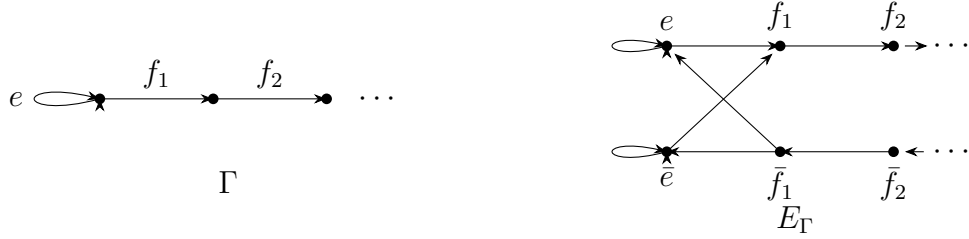
\begin{center}
\begin{tikzpicture}[>=Stealth, baseline=(current bounding box.center)]
\fill (0,0) circle (2pt);
\fill (1.5,0) circle (2pt);
\fill (3,0) circle (2pt);
\draw[->] (0,0) edge[loop left, min distance=12mm, looseness=14] node[left] {$e$} (0,0);
\draw[->] (0,0) -- node[above] {$f_1$} (1.5,0);
\draw[->] (1.5,0) -- node[above] {$f_2$} (3,0);
\node at (3.7,0) {$\cdots$};
\node at (1.7,-1.1) {$\Gamma$};
\begin{scope}[xshift=7.5cm]
\fill (0,0.7) circle (2pt) node[above=2pt] {$e$};
\fill (1.5,0.7) circle (2pt) node[above=2pt] {$f_1$};
\fill (3,0.7) circle (2pt) node[above=2pt] {$f_2$};
\node at (3.8,0.7) {$\cdots$};
\fill (0,-0.7) circle (2pt) node[below=2pt] {$\bar e$};
\fill (1.5,-0.7) circle (2pt) node[below=2pt] {$\bar f_1$};
\fill (3,-0.7) circle (2pt) node[below=2pt] {$\bar f_2$};
\node at (3.8,-0.7) {$\cdots$};
\draw[->] (0,0.7) edge[loop left, min distance=10mm, looseness=14] (0,0.7);
\draw[->] (0,-0.7) edge[loop left, min distance=10mm, looseness=14] (0,-0.7);
\draw[->] (0,0.7) -- (1.5,0.7);
\draw[->] (1.5,0.7) -- (3,0.7);
\draw[->] (3.15,0.7) -- (3.45,0.7);
\draw[->] (3,-0.7) -- (1.5,-0.7);
\draw[->] (3.4,-0.7) -- (3.15,-0.7);
\draw[->] (1.5,-0.7) -- (0,-0.7);
\draw[->] (1.5,-0.7) -- (0.1,0.6);
\draw[->] (0,-0.7) -- (1.4,0.6);
\node at (1.7,-1.6) {$E_\Gamma$};
\end{scope}
\end{tikzpicture}
\end{center}
Note that once again $C^*(\Gamma)$ is neither simple nor purely infinite, even though $\Gamma$ is
essential; nor is it an AT-algebra \textup(see Theorem~\ref{thm:genus1}(2)\textup).
\end{example}

\subsection*{Non-essential $g\ge 2$} We illustrate Theorem~\ref{thm:nonessential-K} with a worked
example. Recall the K-theory recipe used in its proof: for a
row-finite directed graph $E$ with sinks, let $W$ denote the set of sinks of $E$ and
$V=E^0\setminus W$, so that the vertex matrix of $E$ has block form
$M=\left(\begin{smallmatrix} B & C\\ 0 & 0\end{smallmatrix}\right)$ with respect to the decomposition
$E^0=V\sqcup W$. Defining $K\colon\Z V\to\Z V\oplus\Z W$ by $K(x)=\big((1-B^{t})x,\,-C^{t}x\big)$,
one has $K_0(C^*(E))\cong\coker K$ and $K_1(C^*(E))\cong\ker K$ \cite[Theorem~3.2]{RS}.

\begin{example}\label{ex:genus-two-nonessential}
Let $\Gamma$ be the non-essential graph of genus two consisting of two loops $a$ and $b$ at a
vertex $v$, together with a pendant edge $e$ from a dead-end $u$ to $v$, shown on the left below
(reversed edges omitted). According to Definition~\ref{def:EGamma}, $E_\Gamma$ has vertices
$\{a,\bar a,b,\bar b,e,\bar e,u_e\}$; the vertex $\bar e$ is a sink, $u_e$ is a source with the
single edge $w_e\colon u_e\to e$, and the remaining edges are the reduced pairs of $\Gamma$, as
shown on the right below.
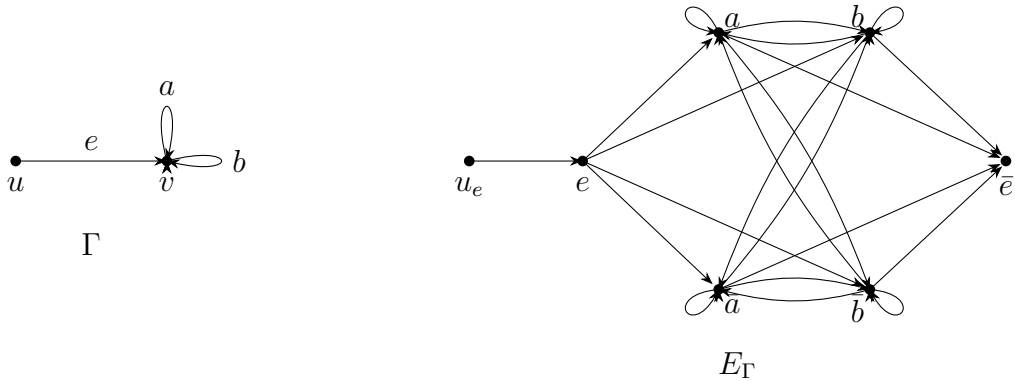
\begin{center}
\begin{tikzpicture}[>=Stealth, baseline=(current bounding box.center)]
\fill (0,0) circle (2pt) node[below=2pt] {$u$};
\fill (2,0) circle (2pt) node[below=2pt] {$v$};
\draw[->] (0,0) -- node[above] {$e$} (2,0);
\draw[->] (2,0) edge[loop above, min distance=10mm, looseness=14] node[above] {$a$} (2,0);
\draw[->] (2,0) edge[loop right, min distance=10mm, looseness=14] node[right] {$b$} (2,0);
\node at (1,-1.1) {$\Gamma$};
\begin{scope}[xshift=7.2cm]
\fill (-1.2,0) circle (2pt) node[below=2pt] {$u_e$};
\fill (0.3,0) circle (2pt) node[below=2pt] {$e$};
\fill (2.1,1.7) circle (2pt) node[above right=-2pt] {$a$};
\fill (4.1,1.7) circle (2pt) node[above left=-2pt] {$b$};
\fill (2.1,-1.7) circle (2pt) node[below right=-2pt] {$\bar a$};
\fill (4.1,-1.7) circle (2pt) node[below left=-2pt] {$\bar b$};
\fill (5.9,0) circle (2pt) node[below=2pt] {$\bar e$};
\draw[->] (-1.2,0) -- (0.3,0);
\draw[->] (0.3,0) -- (2.02,1.63);
\draw[->] (0.3,0) -- (2.02,-1.63);
\draw[->] (0.3,0) -- (4.02,1.66);
\draw[->] (0.3,0) -- (4.02,-1.66);
\draw[->] (2.1,1.7) edge[loop, out=115, in=175, min distance=8mm, looseness=10] (2.1,1.7);
\draw[->] (4.1,1.7) edge[loop, out=65, in=5, min distance=8mm, looseness=10] (4.1,1.7);
\draw[->] (2.1,-1.7) edge[loop, out=185, in=245, min distance=8mm, looseness=10] (2.1,-1.7);
\draw[->] (4.1,-1.7) edge[loop, out=-5, in=-65, min distance=8mm, looseness=10] (4.1,-1.7);
\draw[->] (2.1,1.7) to[bend left=15] (4.1,1.7);
\draw[->] (4.1,1.7) to[bend left=15] (2.1,1.7);
\draw[->] (2.1,-1.7) to[bend left=15] (4.1,-1.7);
\draw[->] (4.1,-1.7) to[bend left=15] (2.1,-1.7);
\draw[->] (2.1,1.7) to[bend left=10] (4.1,-1.7);
\draw[->] (4.1,-1.7) to[bend left=10] (2.1,1.7);
\draw[->] (2.1,-1.7) to[bend right=10] (4.1,1.7);
\draw[->] (4.1,1.7) to[bend right=10] (2.1,-1.7);
\draw[->] (2.1,1.7) -- (5.83,0.04);
\draw[->] (2.1,-1.7) -- (5.83,-0.04);
\draw[->] (4.1,1.7) -- (5.85,0.07);
\draw[->] (4.1,-1.7) -- (5.85,-0.07);
\node at (2.35,-2.7) {$E_\Gamma$};
\end{scope}
\end{tikzpicture}
\end{center}
With respect to $V=(a,\bar a,b,\bar b,e,u_e)$ and $W=(\bar e)$ the blocks of the vertex matrix are
\[
B=\begin{pmatrix}
1&0&1&1&0&0\\
0&1&1&1&0&0\\
1&1&1&0&0&0\\
1&1&0&1&0&0\\
1&1&1&1&0&0\\
0&0&0&0&1&0
\end{pmatrix},
\qquad
C=\begin{pmatrix} 1\\1\\1\\1\\0\\0\end{pmatrix}.
\]
The map $K\colon\Z V\to\Z V\oplus\Z W$ of the recipe above is therefore given by the $7\times 6$
integer matrix
\[
K=\begin{pmatrix}1-B^{t}\\ -C^{t}\end{pmatrix}
=\begin{pmatrix}
0&0&-1&-1&-1&0\\
0&0&-1&-1&-1&0\\
-1&-1&0&0&-1&0\\
-1&-1&0&0&-1&0\\
0&0&0&0&1&-1\\
0&0&0&0&0&1\\
-1&-1&-1&-1&0&0
\end{pmatrix},
\]
whose columns are indexed by $V=(a,\bar a,b,\bar b,e,u_e)$, whose first six rows are the block
$1-B^{t}$ indexed by $V$, and whose final row is the block $-C^{t}$ indexed by the sink $\bar e$.
A direct computation gives
$\ker K=\operatorname{span}_\Z\{(1,-1,0,0,0,0),(0,0,1,-1,0,0)\}$ and $\coker K\cong\Z^{3}$, so by
\cite[Theorem~3.2]{RS}
\[
K_1(C^*(\Gamma))\cong\Z^{2}=\Z^{g},\qquad K_0(C^*(\Gamma))\cong\Z^{3}
\]
(cf.\ Theorem~\ref{thm:nonessential-K}).
\end{example}

\subsection*{Moves on undirected graphs}
Since the genus is the first Betti number, any two connected graphs of the same genus are homotopy
equivalent, and the genus-preserving moves on $\Gamma$ are generated by elementary expansions and
collapses: subdivision of an edge, vertex splitting, and the attachment or pruning of trees. The
first two preserve the Morita class of $C^*(\Gamma)$: subdivision and vertex splitting are inverse
to the contraction of a non-loop edge whose endpoints are not dead-ends, which, by the proof of
Theorem~\ref{thm:K0-torsion}, is implemented on $E_\Gamma$ by the collapse move of
\cite[Theorem~5.2]{Sorensen} and so preserves the stable isomorphism class; these moves preserve
$\gamma$ and $d$ as well as $g$. Tree attachment, by contrast, preserves the genus but changes the
algebra: the bouquet $B_g$ has $K_0\cong\Z^{g}\oplus\Z/(g-1)\Z$
\textup(Theorem~\ref{thm:K0-torsion}\textup), while $B_g$ with a single pendant edge attached has
$K_0\cong\Z^{g+1}$ \textup(Theorem~\ref{thm:nonessential-K} with $\gamma=0$, $d=1$\textup), and
attaching an infinite ray instead gives $\Z^{g+1}$ through the graph valency. Thus the Morita
class of $C^*(\Gamma)$ remembers precisely the triple $(g,\gamma,d)$, together with essentiality
and finiteness, and not merely the homotopy type of $\Gamma$.

For finite essential graphs of genus $g\ge 2$ the genus alone rules: $\gamma=d=0$ is forced, the
algebras are unital Kirchberg algebras, and the Kirchberg--Phillips classification \cite{Phillips}
shows that every genus-preserving move between such graphs preserves the stable isomorphism class.
Isomorphism is finer: the class of the unit in $K_0$ has order $(g-1)/\gcd(g-1,|\Gamma^0|)$
\cite{CLM}, so subdivision changes the isomorphism class while preserving the stable one. Finally,
the in- and out-splittings of \cite{Sorensen} and the Cuntz splice act on $E_\Gamma$ but in
general destroy the dual-graph form, so they have no undirected counterparts; whether the moves
above generate stable isomorphism within the class of $C^*$-algebras of undirected graphs, as the
moves of \cite{Sorensen} do for simple graph algebras, is an interesting open question.

\subsection*{The position of the unit}
When $\Gamma$ is finite the algebra $C^*(\Gamma)$ is unital, with $1=\sum_{v\in\Gamma^0}p_v$, and
the class of the unit refines the stable classification above. Since (R1) and (R2) combine to give
$p_v=\sum_{r(e)=v}s_es_e^{*}$, we have $1=\sum_{e\in\Gamma^1}s_es_e^{*}$, so under the
identification of Corollary~\ref{cor:ker-coker} the unit corresponds to the class of the all-ones
vector $\mathbf{1}\in\Z\Gamma^1$ in $\coker(1-T_\Gamma^{t})$, exactly as for Cuntz--Krieger
algebras. For essential $\Gamma$ of genus $g\ge 2$ the class $[1]$ is torsion: for the bouquet
$B_g$ the computation in the proof of Theorem~\ref{thm:K0-torsion} gives
$[\mathbf{1}]=\sum_ip_i=gp=p$, the generator of the torsion summand $\Z/(g-1)\Z$, and in general
$[1]$ has order $(g-1)/\gcd(g-1,|\Gamma^0|)$ \cite{CLM}. In particular subdivision, which
preserves the stable isomorphism class by the preceding subsection, changes the order of the unit
and hence the isomorphism class of these Kirchberg algebras.

In the stably finite cases the unit is instead a strictly positive vector carrying dimension data
that $K_0$ alone forgets. For a finite tree \textup(Lemma~\ref{lem:branching-ends}(2)\textup) the
unit is the dimension vector $(n_1,\dots,n_d)$ of the matrix summands, where $n_j$ is the number
of vertices of $E_\Gamma$ reaching the $j$-th sink: in the finite case of
Example~\ref{ex:genus-zero-tripod}, $[1]=(5,5,5)$. For a circuit of length $n$
\textup(Theorem~\ref{thm:genus1}(1)\textup) the unit is $(n,n)$ with respect to the rank-one
generators of the two summands, so the triple $(K_0,K_1,[1])$ remembers the circuit length, which
the stable classification cannot. For finite non-essential $\Gamma$ of genus one
\textup(Corollary~\ref{cor:genus-one-nonessential}\textup) the group $K_0\cong\Z^{1+d}$ is free,
so $[1]$ is never torsion---the pendant trees that kill the torsion also free the unit. For
example, for a single loop $a$ at $v$ with one pendant edge attached, the presentation \textup($**$\textup)
gives $[\bar e]=0$ and $[e]=[u_e]=[a]+[\bar a]$, so summing the vertex classes of $E_\Gamma$
yields $[1]=3\big([a]+[\bar a]\big)$.

\subsection*{Summary} Collecting the essential-case results and their non-essential counterparts
(the latter obtained from the directed model) gives the following
picture, in which each cell records $(K_0,K_1)$ and the algebra type, together with a reference to
where the entry is proved. Here $\gamma$ is the graph valency and $d$ the number of dead-ends.

\begin{center}
\renewcommand{\arraystretch}{1.4}
\small
\begin{tabular}{llll}
\hline
& genus & finite & infinite \\
\hline
\multirow{3}{*}{essential}
& $0$ & ---\ {\scriptsize Lem~\ref{lem:branching-ends}(1)} & \begin{tabular}[t]{@{}l@{}}$(\Z^{\gamma},0)$, AF\\ {\scriptsize Lem~\ref{lem:branching-ends}(3), Thm~\ref{thm:genus0}}\end{tabular} \\
& $1$ & \begin{tabular}[t]{@{}l@{}}$(\Z^{2},\Z^{2})$, AT\\ {\scriptsize Thm~\ref{thm:genus1}(1)}\end{tabular} & \begin{tabular}[t]{@{}l@{}}$(\Z^{\gamma+1},\Z)$, not AT\\ {\scriptsize Thm~\ref{thm:genus1}(2)}\end{tabular} \\
& $2\le g<\infty$ & \begin{tabular}[t]{@{}l@{}}$(\Z^{g}\oplus\Z/(g-1),\Z^{g})$, p.i.\\ {\scriptsize Thms~\ref{thm:K1-homology}, \ref{thm:K0-torsion}; Cor~\ref{cor:ker-coker}}\end{tabular} & \begin{tabular}[t]{@{}l@{}}$(\Z^{\gamma+g},\Z^{g})$, p.i.\\ {\scriptsize Thm~\ref{thm:K1-homology}; \cite{Iyudu}}\end{tabular} \\
\hline
\multirow{3}{*}{non-essential}
& $0$ & $(\Z^{d},0)$, f.d.\ {\scriptsize Lem~\ref{lem:branching-ends}(2)} & $(\Z^{\gamma+d},0)$, AF\ {\scriptsize Lem~\ref{lem:branching-ends}(3)} \\
& $1$ & \begin{tabular}[t]{@{}l@{}}$(\Z^{1+d},\Z)$, not AT\\ {\scriptsize Cor~\ref{cor:genus-one-nonessential}}\end{tabular} & \begin{tabular}[t]{@{}l@{}}$(\Z^{1+\gamma+d},\Z)$, not AT\\ {\scriptsize Cor~\ref{cor:genus-one-nonessential}}\end{tabular} \\
& $2\le g<\infty$ & \begin{tabular}[t]{@{}l@{}}$(\Z^{g+d},\Z^{g})$, not AT\\ {\scriptsize Thm~\ref{thm:nonessential-K}}\end{tabular} & \begin{tabular}[t]{@{}l@{}}$(\Z^{g+\gamma+d},\Z^{g})$, not AT\\ {\scriptsize Thm~\ref{thm:nonessential-K}}\end{tabular} \\
\hline
\end{tabular}
\end{center}

\noindent Here ``f.d.'' abbreviates ``finite-dimensional'' and ``p.i.'' abbreviates ``purely infinite simple'' (and the genus-$\ge 2$ algebras are unital
Kirchberg algebras in the finite case, classified by their K-theory). The genus $g=\infty$ case is
not treated.

\pagebreak[4]

\subsection*{Acknowledgement}
This paper grew out of joint work with N.~Brownlowe, A.~Mundey, J.~Spielberg and A.~Thomas, whom the author thanks for
many fruitful discussions. The author acknowledges the help of the AI Claude in pulling results from old notes and pointing
out where errors occurred. Of course, Claude helped with the layout, grammar and spelling of the
paper. A copy of the prompts has been saved for archival purposes. No data was generated as part
of this project.


\begin{thebibliography}{99}

\bibitem{Bates} T.~Bates, \emph{Applications of the gauge-invariant uniqueness theorem for graph algebras}, Bull. Austral. Math. Soc. \textbf{65} (2002), 57--67.

\bibitem{BMPST} N.~Brownlowe, A.~Mundey, D.~Pask, J.~Spielberg, A.~Thomas, \emph{$C^*$-algebras
associated to graphs of groups}, Adv. Math. \textbf{316} (2017), 114--186.

\bibitem{CLM} G.~Cornelissen, O.~Lorscheid, M.~Marcolli, \emph{On the K-theory of graph
$C^*$-algebras}, Acta Appl. Math. \textbf{102} (2008), 57--69.

\bibitem{CK} J.~Cuntz, W.~Krieger, \emph{A class of $C^*$-algebras and topological Markov chains},
Invent. Math. \textbf{56} (1980), 251--268.

\bibitem{Iyudu} N.~Iyudu, \emph{K-theory of locally finite graph $C^*$-algebras}, J. Geom. Phys.
\textbf{71} (2013), 22--29.

\bibitem{KPR} A.~Kumjian, D.~Pask, I.~Raeburn, \emph{Cuntz--Krieger algebras of directed graphs},
Pacific J. Math. \textbf{184} (1998), 161--174.

\bibitem{PR} D.~Pask, I.~Raeburn, \emph{On the K-theory of Cuntz--Krieger algebras}, Publ. Res.
Inst. Math. Sci. \textbf{32} (1996), 415--443.

\bibitem{PRennie} D.~Pask, A.~Rennie, \emph{The noncommutative geometry of graph $C^*$-algebras I:
the index theorem}, J. Funct. Anal. \textbf{233} (2006), 92--134.

\bibitem{Phillips} N.~C.~Phillips, \emph{A classification theorem for nuclear purely infinite
simple $C^*$-algebras}, Doc. Math. \textbf{5} (2000), 49--114.

\bibitem{Raeburn} I.~Raeburn, \emph{Graph Algebras}, CBMS Regional Conference Series in Mathematics,
vol.~103, Amer. Math. Soc., Providence, RI, 2005.

\bibitem{RS} I.~Raeburn, W.~Szyma\'nski, \emph{Cuntz--Krieger algebras of infinite graphs and
matrices}, Trans. Amer. Math. Soc. \textbf{356} (2004), 39--59.

\bibitem{Robertson} G.~Robertson, \emph{Invariant boundary distributions for finite graphs},
J. Combin. Theory Ser. A \textbf{115} (2008), 1272--1278.

\bibitem{Rordam} M.~R\o rdam, \emph{Classification of nuclear, simple $C^*$-algebras}, in
Classification of Nuclear $C^*$-Algebras. Entropy in Operator Algebras, Encyclopaedia Math. Sci.,
vol.~126, Springer, Berlin, 2002, pp.~1--145.

\bibitem{Serre} J.-P.~Serre, \emph{Trees}, Springer-Verlag, Berlin, 1980.

\bibitem{Sorensen} A.~P.~W.~S\o rensen, \emph{Geometric classification of simple graph algebras},
Ergodic Theory Dynam. Systems \textbf{33} (2013), 1199--1220.

\end{thebibliography}
\end{document}